\documentclass[a4paper,11pt,]{article}
\usepackage{}
\usepackage{amsfonts}
\usepackage{graphics}
\usepackage{amsthm}
\usepackage{hyperref}
\usepackage{mathrsfs}
\usepackage{fancyhdr}
\usepackage{mathrsfs}
\usepackage{amsthm}
\usepackage{mathrsfs}
\usepackage{amsfonts}
\usepackage{amsmath}
\usepackage{amssymb}
\usepackage{amsthm}
\usepackage{enumerate}
\usepackage[mathscr]{eucal}
\usepackage{eqlist}
\usepackage{graphicx}
\usepackage{color}
\usepackage[marginal]{footmisc}
\usepackage[body={16.7true cm, 24.9true cm}]{geometry}

\newtheorem{Theorem}{Theorem}[section]

\newtheorem{Lemma}[Theorem]{Lemma}
\newtheorem{Corollary}[Theorem]{Corollary}
\newtheorem{Remark}[Theorem]{Remark}

\numberwithin{equation}{section} \allowdisplaybreaks
\usepackage{mathdots}

\renewcommand\abstract{{\bf Abstract}}

\begin{document}
\title{Critical points of solutions to a kind of linear elliptic equations in multiply connected domains\footnote{\footnotesize The work is supported by The work is supported by National Natural Science Foundation of China (No.11401310) and Postgraduate Research \& Practice Innovation Program of Jiangsu Province (KYCX17\_0321). The first author is fully supported by China Scholarship Council(CSC) for visiting Rutgers University (201806840122).}}

\author{Haiyun Deng$^{1}$\footnote{E-mail: haiyundengmath1989@163.com(H. Deng), hrliu@njfu.edu.cn(H. Liu), xpyang@nju.edu.cn(X. Yang)}, Hairong Liu$^{2}$, Xiaoping Yang$^{3}$\\[12pt]
\small \emph {$^{1}$School of Science, Nanjing University of Science and Technology, Nanjing, Jiangsu, 210094, China;}\\
\small \emph {$^{2}$School of Science, Nanjing Forestry University, Nanjing, Jiangsu, 210037, China;}\\
\small \emph {$^{3}$Department of Mathematics, Nanjing University, Nanjing, Jiangsu, 210093, China}}
\date{}
\maketitle

\renewcommand{\labelenumi}{[\arabic{enumi}]}

\begin{abstract}{\bf:} \footnotesize
 In this paper, we mainly study the critical points and critical zero points of solutions $u$ to a kind of linear elliptic equations with nonhomogeneous Dirichlet boundary conditions in a multiply connected domain $\Omega$ in $\mathbb{R}^2$. Based on the fine analysis about the distributions of connected components of the super-level sets $\{x\in \Omega: u(x)>t\}$ and sub-level sets $\{x\in \Omega: u(x)<t\}$ for some $t$, we obtain the geometric structure of interior critical point sets of $u$. Precisely, let $\Omega$ be a multiply connected domain with the interior boundary $\gamma_I$ and the external boundary $\gamma_E$, where $u|_{\gamma_I}=\psi_1(x),~u|_{\gamma_E}=\psi_2(x)$. When $\psi_1(x)$ and $\psi_2(x)$ have $N_1$ and $N_2$ local maximal points on $\gamma_I$ and $\gamma_E$ respectively, we deduce that $\sum_{i = 1}^k {{m_i}}\leq N_1+ N_2$, where $m_1,\cdots,m_k$ are the respective multiplicities of interior critical points $x_1,\cdots,x_k$ of $u$. In addition, when $\min_{\gamma_E}\psi_2(x)\geq \max_{\gamma_I}\psi_1(x)$ and $u$ has only $N_1$ and $N_2$ equal local maxima relative to $\overline{\Omega}$ on $\gamma_I$ and $\gamma_E$ respectively, we develop a new method to show that one of the following three results holds $\sum_{i = 1}^k {{m_i}}=N_1+N_2$ or $\sum_{i = 1}^k {{m_i}}+1=N_1+N_2$ or $\sum_{i = 1}^k {{m_i}}+2=N_1+N_2$. Moreover, we investigate the geometric structure of interior critical zero points of $u.$ We obtain that the sum of multiplicities of the interior critical zero points of $u$ is less than or equal to the half of the number of its isolated zero points on the boundaries.

\end{abstract}



\section{Introduction and main results}
~~~~~In this paper we mainly investigate the interior critical points and critical zero points of solutions to the following elliptic equations
\begin{equation}\label{1.1}\begin{array}{l}
\begin{array}{l}
Lu:=\sum\limits_{i,j=1}^{2}a_{ij}(x)u_{x_ix_j}(x)+\sum\limits_{i=1}^{2}b_{i}(x)u_{x_i}(x)=0~~~\mbox{in}~\Omega,
\end{array}
\end{array}\end{equation}
where $\Omega$ is a bounded, smooth and multiply connected domain in $\mathbb{R}^{2}$, $a_{ij},b_{i}$ are smooth and $L$ is uniformly elliptic in $\Omega$.

 Critical points of solutions to elliptic equations is a significant research topic. There are many known results about the critical points. In 1987 Alessandrini \cite{Alessandrini} investigated the geometric structure of the critical point sets of solutions to linear elliptic equations with nonhomogeneous Dirichlet boundary condition in a planar bounded simply connected domain. In 1994, Sakaguchi \cite{Sakaguchi2} considered the critical points of solutions to an obstacle problem in a planar bounded smooth simply connected domain. He showed that if the number of critical points of the obstacle is finite and the obstacle has only $N$ local (global) maximum points, then $\Sigma_{i = 1}^k {{m_i}}+1 \le N$ ($\Sigma_{i = 1}^k {{m_i}}+1=N$), where $m_1,m_2,\cdots,m_k$ are the respective multiplicities of the critical points $x_1,x_2,\cdots,x_k$ of $u$ in the noncoincidence set. In 2012 Arango and G\'{o}mez \cite{Arango} studied the critical points of the solutions to a quasilinear elliptic equation with Dirichlet boundary condition in planar strictly convex and nonconvex domains respectively. In 2018, Deng, Liu and Tian \cite{Deng3} investigated the geometric structure of interior critical point sets of solutions $u$ to a quasilinear elliptic equation with nonhomogeneous Dirichlet boundary condition in a simply connected or multiply connected domain $\Omega$ in $\mathbb{R}^2$, and proved that $\Sigma_{i = 1}^k {{m_i}}+1=N$ or $\Sigma_{i = 1}^k {{m_i}}=N,$ where $m_1,\cdots,m_k$ are the respective multiplicities of interior critical points $x_1,\cdots,x_k$ of $u$ and $N$ is the number of global maximal points of $u$ on $\partial\Omega$. For the related research work of critical points, see \cite{AlessandriniMagnanini1,Alessandrini1,AlessandriniMagnanini2,Cecchini,ChenHuang,Deng2,Enciso1,Enciso2,Kraus,Sakaguchi1}.

The geometric structure of critical point sets of solutions to elliptic equations in higher dimensional spaces has been studied by many people. Under the assumption of the existence of a semi-stable solution of Poisson equation $-\triangle u=f(u)$, Cabr\'{e} and Chanillo \cite{CabreChanillo} showed that the solution $u$ has exactly one nondegenerate critical point in bounded smooth convex domains of $\mathbb{R}^{n}(n\geq2)$. In 1998, Han, Hardt and Lin \cite{Han} concerned with the geometric measure of singular sets of weak solutions of second-order elliptic equations. They gave an estimate on the measure of singular sets in terms of the frequency of solutions. In 2015, Cheeger, Naber and Valtorta \cite{Cheeger} introduced some techniques for giving measure estimates of the critical sets, including the Minkowski type estimates on the effective critical set $C_r(u)$. In 2017, Naber and Valtorta \cite{Naber} introduced some techniques for estimating the critical sets and singular sets which avoid the need of any $\varepsilon$-regularity lemmas. In 2017 Alberti, Bal and Di Cristo \cite{Alberti} studied the existence of critical points for solutions to second-order elliptic equations of the form $\nabla\cdot\sigma(x)\nabla u=0$ in a bounded domain with prescribed boundary conditions in $\mathbb{R}^{n}(n\geq3).$ See \cite{Deng1,Donnelly1,Donnelly2,Hardt,Lin,Logunov1,Logunov2,Logunov3,Tian1,Tian2} for some related work.

The goal of this paper is further to study the interior critical points and critical zero points of solutions to a kind of linear elliptic equations with nonhomogeneous Dirichlet boundary conditions in a multiply connected domain. Our main results are as follows.

\begin{Theorem}\label{th1.1} 
 Let $\Omega$ be a bounded, smooth and multiply connected domain with one interior boundary $\gamma_I$ and the external boundary $\gamma_E$ in $\mathbb{R}^{2}$ and $\psi_1(x),\psi_2(x)\in C^1(\overline{\Omega}).$ Suppose that $\psi_1(x)$ and $\psi_2(x)$ have $N_1$ local maximal points and $N_2$ local maximal points on $\gamma_I$ and $\gamma_E$ respectively. Let $u$ be a non-constant solution of the following boundary value problem
\begin{equation}\label{1.2}\begin{array}{l}
\left\{
\begin{array}{l}
\sum\limits_{i,j=1}^{2}a_{ij}(x)u_{x_ix_j}(x)+\sum\limits_{i=1}^{2}b_{i}(x)u_{x_i}(x)=0~~~\rm{in}~\Omega,\\
u|_{\gamma_I}=\psi_1(x),u|_{\gamma_E}=\psi_2(x).
\end{array}
\right.
\end{array}\end{equation}
 Then $u$ has finite interior critical points, denoting by $x_1,x_2,\cdots,x_k$, and the following inequality holds
\begin{equation}\label{1.3}\begin{array}{l}
\begin{array}{l}
\sum\limits_{i = 1}^k {{m_i}}\le N_1+N_2,
\end{array}
\end{array}\end{equation}
where $m_1,m_2,\cdots,m_k$ are the multiplicities of critical points $x_1,x_2,\cdots,x_k$ respectively.
\end{Theorem}

\begin{Theorem}\label{th1.2} 
  Suppose that the domain $\Omega$ satisfies the hypothesis of Theorem \ref{th1.1} and $\min_{\gamma_E} \psi_2(x)\geq \max_{\gamma_I} \psi_1(x).$ Let $u$ be a non-constant solution of (\ref{1.2}). In addition, suppose that $u$ has only $N_1$ equal local maxima and $N_1$ equal local minima relative to $\overline{\Omega}$ on $\gamma_I$ and that $u$ has only $N_2$ equal local maxima and $N_2$ equal local minima relative to $\overline{\Omega}$ on $\gamma_E$, i.e., the values of all local maximal (minimum) points on the corresponding boundary are equal. Then $u$ has finite interior critical points, and one of the following three results holds
\begin{equation}\label{1.4}\begin{array}{l}
\begin{array}{l}
\sum\limits_{i = 1}^k {{m_i}}= N_1+N_2,
\end{array}
\end{array}\end{equation}
or
\begin{equation}\label{1.5}\begin{array}{l}
\begin{array}{l}
\sum\limits_{i = 1}^k {{m_i}}+1= N_1+N_2,
\end{array}
\end{array}\end{equation}
or
\begin{equation}\label{1.6}\begin{array}{l}
\begin{array}{l}
\sum\limits_{i = 1}^k {{m_i}}+2= N_1+N_2,
\end{array}
\end{array}\end{equation}
where $m_i$ is as in Theorem \ref{th1.1}.
\end{Theorem}

\begin{Theorem}\label{th1.3} 
Let $\Omega$ be a bounded, smooth and multiply connected domain with the interior boundary $\gamma_I$ and the external boundary $\gamma_E$ in $\mathbb{R}^{2}$. Suppose that $\psi(x)\in C^1(\overline{\Omega})$ is sign-changing, $H$ is a given constant and that $\psi$ has $\widetilde{N}$ zero points on $\gamma_E.$  Let $u$ be a non-constant solution of the following boundary value problem
\begin{equation}\label{1.7}\begin{array}{l}
\left\{
\begin{array}{l}
\sum\limits_{i,j=1}^{2}a_{ij}(x)u_{x_ix_j}(x)+\sum\limits_{i=1}^{2}b_{i}(x)u_{x_i}(x)=0~~~\rm{in}~\Omega,\\
u|_{\gamma_I}=H,~~u|_{\gamma_E}=\psi(x).
\end{array}
\right.
\end{array}\end{equation}
Then $u$ has finite interior critical zero points, denoting by $x_1,\cdots,x_l$, and\\
if $H\neq 0,$ we have
\begin{equation}\label{1.8}\begin{array}{l}
\begin{array}{l}
\sum\limits_{i = 1}^l {{m_i}}\leq \frac{\widetilde{N}}{2},
\end{array}
\end{array}\end{equation}
else $H=0,$ we have
\begin{equation}\label{1.9}\begin{array}{l}
\begin{array}{l}
\sum\limits_{i = 1}^l {{m_i}}\leq \frac{\widetilde{N}}{2}-1,
\end{array}
\end{array}\end{equation}
where critical zero point $x$ means that $u(x)=|\nabla u(x)|=0$ and $m_i$ is the multiplicity of corresponding critical zero point $x_i$.
\end{Theorem}

\begin{Theorem}\label{th1.4} 
Suppose that the domain $\Omega$ satisfies the hypothesis of Theorem \ref{th1.1}. Let $u$ be a non-constant solution of (\ref{1.2}). In addition, suppose that $\psi_1(x)$ and $\psi_2(x)$ are sign-changing and have $\widetilde{N}_1$ and $\widetilde{N}_2$ zero points on $\gamma_I$ and $\gamma_E$ respectively. Then $u$ has finite interior critical zero points and
\begin{equation}\label{1.10}\begin{array}{l}
\begin{array}{l}
\sum\limits_{i = 1}^l {{m_i}}\leq \frac{\widetilde{N}_1+\widetilde{N}_2}{2},
\end{array}
\end{array}\end{equation}
where $m_i$ is as in Theorem \ref{th1.3}.
\end{Theorem}

\begin{Remark}\label{re1.5}
According to the assumptions of Theorem \ref{th1.2} and the strong maximum principle, we know that the local minima of $u$ on $\gamma_I$ and the local maxima of $u$ on $\gamma_E$ are the global minima and global maxima of $u$ on $\overline{\Omega}$ respectively.
\end{Remark}

\begin{Remark}\label{re1.6}
If the assumption $\mathop {\min}_{\gamma_E} \psi_2(x)\geq \mathop {\max}_{\gamma_I} \psi_1(x)$ is replaced by $\mathop {\max}_{\gamma_E} \psi_2(x)\leq \mathop {\min}_{\gamma_I} \psi_1(x)$ in Theorem \ref{th1.2}, the conclusions are still valid. In addition, if the elliptic operator $L$ in (\ref{1.1}) is replaced by $Lu:=\sum_{i,j=1}^{2}a_{ij}(\nabla u)u_{x_ix_j}(x)=0,$ where $a_{ij}$ is smooth and L is uniformly elliptic in $\Omega,$ then the conclusions of Theorem \ref{th1.1}, Theorem \ref{th1.2}, Theorem \ref{th1.3} and Theorem \ref{th1.4} still hold.
\end{Remark}

\begin{Remark}\label{re1.7}
If the elliptic operator $L$ in (\ref{1.1}) is replaced by $Lu:=\sum_{i,j=1}^{2}a_{ij}(x)u_{x_ix_j}(x)+\sum_{i=1}^{2}b_{i}(x)u_{x_i}(x)+c(x)u(x)=0,$ where $c(x)$ is smooth, $c(x)\leq 0$ in $\Omega$ and $a_{ij},b_i$ is as in (\ref{1.1}). By the results of \cite{Han}, we know that the interior critical zero points are isolated. Then the conclusions of Theorem \ref{th1.3} and Theorem \ref{th1.4} still hold.
\end{Remark}

Throughout this paper, we set $z_1:=\mathop{\min }_{\gamma_I}\psi_1(x),~Z_1:=\mathop{\max }_{\gamma_I}\psi_1(x),$ $z_2:=\mathop{\min }_{\gamma_E}\psi_2(x),~Z_2:=\mathop{\max }_{\gamma_E}\psi_2(x).$ For the sake of clarity, we now explain the key ideas which are used to prove the main results.  Based on the fine analysis about the distributions of connected components of some level sets, we prove (\ref{1.3}) by induction and the strong maximum principle. On the other hand, we develop a new method to prove (\ref{1.4}), (\ref{1.5}) and (\ref{1.6}), which is different from the methods in \cite{Alessandrini,Sakaguchi2}. Concerning the case of simply connected domains. In \cite{Alessandrini}, the author proved the result by induction on the number $N$ of global maximal points on $\partial\Omega$. In \cite{Sakaguchi2}, the author showed the result by using induction and a differential homeomorphism method.

Actually, when $z_2\geq Z_1$ and $u$ has only $N_1$ ($N_2$) equal local maxima and $N_1$ ($N_2$) equal local minima relative to $\overline{\Omega}$ on $\gamma_I$ ($\gamma_E$), we will prove that all critical values are equal to one of the two values, i.e., $u(x_{1})=\cdots=u(x_l)=t_1$ and $u(x_{l+1})=\cdots=u(x_k)=t_0$ for some $t_1\neq t_0$. We obtain (\ref{1.4}), (\ref{1.5}) or (\ref{1.6}) by showing that there are the following three ``{\bf just right}''s similar to those in \cite{Deng3}:

 (i) the first ``{just right}'' is that all critical values are equal to one of the two values;

 (ii) the second  ``{just right}'' is that  critical points together with the corresponding level lines of $\{x\in \Omega: u(x)=t_1~(t_0)~\mbox{for some}~t_1\in (z_2,Z_2)~(t_0\in (z_1,Z_1))\}$, which meet $\gamma_E$ ($\gamma_I$), clustering round these points exactly form one connected set;

 (iii) the third ``{just right}'' is that  every connected component of $\{x\in \Omega: u(x)>t_1>z_2\}$ ($\{x\in \Omega: u(x)<t_0<Z_1\}$) has exactly one global maximal (minimum) point on $\gamma_E$ ($\gamma_I$).

The rest of this paper is organized as follows. In Section 2, we give the fine analysis about the distributions of connected components of super-level sets $\{x\in \Omega: u(x)>t\}$ and sub-level sets $\{x\in \Omega: u(x)<t\}$ for some $t$. Among others, we will prove Theorems \ref{th1.1} and \ref{th1.2} in Sections 3 and 4 respectively, and Theorems \ref{th1.3} and \ref{th1.4} in Section 5.

\section{Preliminaries}
~~~~~According to the assumption of Theorem \ref{th1.1}, we know that one of the following three occurs: (I) the global maximal points are only on $\gamma_E$, (II) the global maximal points are only on $\gamma_I$, (III) the global maximal points are on $\gamma_E$ and $\gamma_I$ at the same time.

For (I), we have one of the following three cases
$$z_1<Z_1\leq z_2<Z_2,~z_1<z_2<Z_1<Z_2~\mbox{and}~z_2\leq z_1<Z_1<Z_2.$$
One can easily observe that the third case is similar to the second one from the proofs of Lemmas in this section. Moreover  during the course of proofs we also know that (II) is obviously similar to (I) and (III) is simpler than (I). In fact (I) has three cases and (III) has only two cases $z_1 \leq z_2 < Z_1 = Z_2$ and $z_2 <z_1 < Z_1 = Z_2.$ Furthermore, there is a case $z_1 < z_2 < Z_1 < Z_2$ in (I), which is the most complex case
that we discuss. Thus, from later on we only describe the lemmas and prove them for Case (I). Our main idea is to sort the size of the four values $Z_2, z_2, Z_1, z_1,$ and then study the distribution of the super-level (sub-level) sets of the corresponding interval.

In order to prove Theorem \ref{th1.1} for the cases of $z_1<Z_1\leq z_2<Z_2$ and $z_1<z_2<Z_1<Z_2$, we need the following lemmas.

\begin{Lemma}\label{le2.1}
 Suppose that $x_0$ is an interior critical point of $u$ in $\Omega$ and that $m$ is the multiplicity of $x_0.$ Then there exist $m+1$ distinct connected components of $\{x\in \Omega :u(x)>u(x_0)\}$ and $\{x\in \Omega :u(x)<u(x_0)\}$ clustering around the point $x_0$ respectively.
 \begin{proof}[Proof] By the results of Hartman and Wintner \cite{Hartman}, in a neighborhood of $x_0$, the level line $\{x\in \Omega : u(x)=u(x_0)\}$ consists of $m+1$ simple arcs intersecting at $x_0$. Then the conclusion can be easily obtained.
\end{proof}
\end{Lemma}

\begin{Lemma}\label{le2.2} 
(i) Suppose that $z_1<Z_1\leq z_2<Z_2$ and $u$ is a non-constant solution of (\ref{1.2}), then we have:

(1) For any  $t\in (z_2,Z_2),$ any connected component $\omega$ of $\{x\in \Omega : u(x)>t\}$ has to meet the external boundary $\gamma_E.$

(2) For any  $t\in (z_1,Z_1),$ any connected component $\omega$ of $\{x\in \Omega : u(x)<t\}$ has to meet the interior boundary $\gamma_I.$\\
(ii) Suppose that $z_1<z_2<Z_1<Z_2$ and $u$ is a non-constant solution of (\ref{1.2}), then we have:

(3) For any  $t\in [Z_1,Z_2),$ any connected component $\omega$ of $\{x\in \Omega : u(x)>t\}$ has to meet the external boundary $\gamma_E.$

(4) For any  $t\in (z_2,Z_1),$ any connected component $\omega$ of $\{x\in \Omega : u(x)>t\}$ or $\{x\in \Omega : u(x)<t\}$ may meet $\gamma_I$ or $\gamma_E.$

(5) For any  $t\in (z_1,z_2],$ any connected component $\omega$ of $\{x\in \Omega : u(x)<t\}$ has to meet the interior boundary $\gamma_I.$
\end{Lemma}
\begin{proof}[Proof](1) For any $t\in (z_2,Z_2),$ suppose that $A$ is a connected component of $\{x\in \Omega : u(x)>t\}$.  According to the assumption of $u|_{\gamma_I}=\psi_1(x)$ and $z_2\geq Z_1$, we know that $A$ can not contain the interior boundary $\gamma_I$. Then the strong maximum principle shows that the connected component $A$ has to meet the external boundary $\gamma_E.$ The proofs of (2), (3) and (5) are same as the proof of (1) and the results of (4) naturally hold.
\end{proof}

\begin{Lemma}\label{le2.3}
Let $u$ be a non-constant solution to (\ref{1.2}). Then $u$ has finite interior critical points in $\Omega$.
\begin{proof}[Proof] We set up the usual contradiction argument. Since the theorem of Hartman and Wintner \cite{Hartman} shows that the interior critical points of $u$ are isolated, so we suppose that $u$ has infinite interior critical points in $\Omega,$ denoting by $x_1,x_2,\cdots$. The results of Lemma \ref{le2.1} and Lemma \ref{le2.2} show that there exist infinite connected components of $\{x\in \Omega :u(x)>u(x_i)\}$ and $\{x\in \Omega :u(x)<u(x_i)\}~(i=1,2,\cdots)$, which meet the boundary $\gamma_E$ or $\gamma_I$. The strong maximum principle implies that there totally exist infinite local maximal points and minimum points on $\gamma_E$ and $\gamma_I$, this contradicts with the assumption of Theorem \ref{th1.1}.
\end{proof}
\end{Lemma}

\begin{Lemma}\label{le2.4}
Suppose that $z_1<Z_1\leq z_2<Z_2$ and $u$ is a non-constant solution to (\ref{1.2}). Then there does not exist any interior critical point $x$ such that $u(x)=t$ for any $t\in [Z_1,z_2].$
\begin{proof}[Proof] We divide the proof into two cases.

(i) Case 1: When $Z_1=z_2$, without loss of generality, we suppose by contradiction that there exists an interior critical point $x_0$ such that $u(x_0)=Z_1=z_2$ and that the multiplicity of $x_0$ is one. By Lemma \ref{le2.1} and Lemma \ref{le2.2}, one of the distributions for $x_0$ and the corresponding level lines of $\{x\in\Omega : u(x)=z_2\}$ is show in Fig. 1.

\begin{center}
  \includegraphics[width=16.0cm,height=3.2cm]{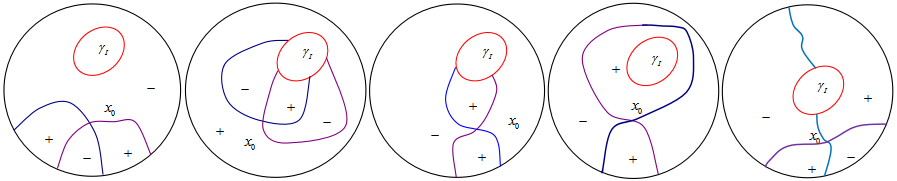}\\
  \scriptsize {\bf Fig. 1.}~~The distributions for $x_0$ and the corresponding level lines of $\{x\in\Omega : u(x)=z_2\}$.
\end{center}
Any one of the distributions of Fig. 1 contradicts with $Z_1=z_2<Z_2$ and the continuity of $u$.

(ii) Case 2: When $Z_1<z_2$, without loss of generality, we suppose by contradiction that there exists an interior critical point $x_0$ such that $u(x_0)=t_0$ for $t_0=Z_1$ or $t_0\in (Z_1,z_2)$ or $t_0=z_2$ and that the multiplicity of $x_0$ is one. The proof is same as Case 1.
\end{proof}
\end{Lemma}

\begin{Lemma}\label{le2.5}
 Let $z_1<Z_1\leq z_2<Z_2$ and $x_1,x_2,\cdots,x_k$ be the interior critical points of $u$ in $\Omega$. Suppose that $z_1<u(x_{l+1})=\cdots=u(x_k)=t_0<Z_1\leq z_2<u(x_1)=\cdots=u(x_l)=t_1<Z_2$ and that $x_1,\cdots,x_l$ and $x_{l+1},\cdots,x_k$ together with the corresponding level lines of $\{x\in\Omega : u(x)=t_1\}$ and $\{x\in\Omega : u(x)=t_0\}$ clustering round these points form $q_1,q_0$ connected sets respectively, where $q_1,q_0\geq 1$ and $m_1,m_2,\cdots,m_k$ are the multiplicities of critical points $x_1,x_2,\cdots,x_k$ respectively.

Case 1: Suppose that there exists a non-simply connected component $\omega$ of $\{x\in \Omega : u(x)<t_1\}$ and the external boundary $\gamma$ of $\omega$ is a simply closed curve between $\gamma_I$ and $\gamma_E$ such that $u$ has at least one critical point on $\gamma$, then
 \begin{equation}\label{2.1}\begin{array}{l}
\sharp\Big\{\mbox{the\ simply\ connected\ components\ $\widetilde{\omega}$ of\ the\ sub-level\ set} ~ \{x\in\Omega : u(x)<t_1\}\\~~~ \mbox{such\ that\ $\widetilde{\omega}$ meet\ the\ external\ boundary\ $\gamma_E$} \Big\}= \sum\limits_{i = 1}^l {{m_i}}+q_1-1.
\end{array}\end{equation}

Case 2: Suppose that there exists a non-simply connected component $\omega$ of $\{x\in \Omega : u(x)<t_1\}$ such that $\omega$ meets $\gamma_E$. In addition, we set $M_1$ and $M_2$ as the number of the connected components of the super-level set $\{x\in\Omega : u(x)>t_1\}$ and the sub-level set $\{x\in\Omega : u(x)<t_1\}$, respectively. Then
\begin{equation}\label{2.2}\begin{array}{l}
M_1\geq \sum\limits_{i = 1}^l {{m_i}}+1,~M_2\geq \sum\limits_{i = 1}^l {{m_i}}+1, ~\mbox{and}~M_1+M_2=2\sum\limits_{i = 1}^l {{m_i}}+q_1+1.
\end{array}\end{equation}

Case 3: Suppose that there exists a simply closed curve $\gamma$ of $\{x\in \Omega : u(x)=t_0\}$ between $\gamma_I$ and $\gamma_E$ such that $u$ has at least one critical point on $\gamma$, then
 \begin{equation}\label{2.3}\begin{array}{l}
\sharp\Big\{\mbox{the\ simply\ connected\ components\ $\omega$ of\ the\ super-level\ set} ~ \{x\in\Omega : u(x)>t_0\}\\~~~ \mbox{such\ that\ $\omega$ meet\ the\ interior\ boundary\ $\gamma_I$} \Big\}= \sum\limits_{i = l+1}^k {{m_i}}+q_0-1.
\end{array}\end{equation}

Case 4: For $t_0\in (z_1,Z_1)$, if Case 3 does not occur, we set $\widetilde{M}_1$ and $\widetilde{M}_2$ as the number of the connected components of $\{x\in\Omega : t_0<u(x)<z_2\}$ and the sub-level set $\{x\in\Omega : u(x)<t_0\}$, respectively. Then
\begin{equation}\label{2.4}\begin{array}{l}
\widetilde{M}_1\geq \sum\limits_{i = l+1}^k {{m_i}}+1,~\widetilde{M}_2\geq \sum\limits_{i = l+1}^k {{m_i}}+1, ~\mbox{and}~\widetilde{M}_1+\widetilde{M}_2=2\sum\limits_{i = l+1}^k {{m_i}}+q_0+1.
\end{array}\end{equation}
\end{Lemma}

\begin{proof}[Proof] (i) Case 1: We divide the proof into two steps.

  Step 1: When $q_1=1,$ by induction. When $l=1,$ the result holds by Lemma \ref{le2.1} and Lemma \ref{le2.2}. For $1\leq l\leq n$, we assume that the connected set, which consists of $l$ critical points and the connected components clustering round these points, contains exactly $\sum_{i = 1}^l {{m_i}}$ components $\omega$ of the sub-level set $\{x\in\Omega : u(x)<t_1\}$ such that $\omega$ meet $\gamma_E$. Let $l=n+1.$ Let $A$ be the set which consists of the points $x_1,x_2,\cdots,x_{n+1}$ together with the respective components clustering round these points. We may assume that the points $x_1,x_2,\cdots,x_n$ together with the respective components clustering round these points form a connected set, denotes by $B.$ By Lemma \ref{le2.2}, we know that $A$ cannot surround a component of $\{x\in\Omega : u(x)>t_1\}.$  Up to renumbering, therefore there is only one component of $\{x\in\Omega : u(x)<t_1\}$ whose boundary $\alpha$ contains both $x_n$ and $x_{n+1}$. The distribution for the level lines of $\{x\in\Omega : u(x)=t_1\}$ is shown in Fig. 2.
\begin{center}
  \includegraphics[width=6cm,height=3.6cm]{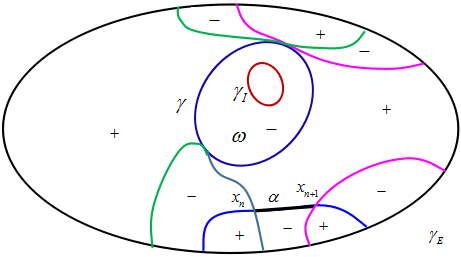}\\
  \scriptsize {\bf Fig. 2.}~~The distribution for the level lines of $\{x\in\Omega : u(x)=t_1\}$.
\end{center}
 Noting that both $A$ and $B$ are connected, by using Lemma \ref{le2.1} and the inductive assumption to $B,$ we know that $A$ contains exactly
 $$ \sum\limits_{i = 1}^n {{m_i}}+(m_{n+1}+1)-1=\sum\limits_{i = 1}^{n+1} {{m_i}}$$
connected components $\widetilde{\omega}$ of the sub-level set $\{x\in\Omega : u(x)<t_1\}$ such that $\widetilde{\omega}$ meet $\gamma_E$. This completes the proof of step 1.

Step 2: When $q_1\geq 2.$ We easily know that the number of connected sets of the level lines $\{x\in\Omega : u(x)=t_1\}$ together with $x_1,\cdots,x_l$ increases one leading the number of connected components of $\{x\in\Omega : u(x)<t_1\}$ increases one. If all the critical points $x_1,x_2,\cdots,x_l$ together with the level lines $\{x\in\Omega : u(x)=t_1\}$ clustering round these points form $q_1$ connected sets. By the results of step 1, then we have
\begin{equation*}\begin{array}{l}
 \sharp\Big\{\mbox{the\ simply\ connected\ components\ $\widetilde{\omega}$ of\ the\ sub-level\ set} ~ \{x\in\Omega : u(x)<t_1\}\\~~~ \mbox{such\ that\ $\widetilde{\omega}$ meet\ the\ external\ boundary\ $\gamma_E$} \Big\}= \sum\limits_{i = 1}^l {{m_i}}+(q_1-1).
\end{array}\end{equation*}
This completes the proof of Case 1.

(ii) Case 2: We divide the proof of Case 2 into two steps.

  Step 1: When $q_1=1$, by induction. When $l=1,$ the result holds by Lemma \ref{le2.1} and Lemma \ref{le2.2}. For $1\leq l\leq n$, we assume that the connected set, which consists of $l$ critical points and the connected components clustering round these points, contains exactly $\sum\limits_{i = 1}^l {{m_i}}+1$ components of the super-level set $\{x\in\Omega : u(x)>t_1\}$. Let $l=n+1.$ Let $A$ be the set which consists of the points $x_1,x_2,\cdots,x_{n+1}$ together with the respective components clustering round these points. We may assume that the points $x_1,x_2,\cdots,x_n$ together with the respective components clustering round these points form a connected set, denotes by $B.$ By Lemma \ref{le2.2}, we know that $A$ cannot surround a component of $\{x\in\Omega : u(x)<t_1\}.$  Up to renumbering, therefore there is only one component of $\{x\in\Omega : u(x)>t_1\}$ whose boundary $\gamma$ contains both $x_n$ and $x_{n+1}$. The distribution for the level lines of $\{x\in\Omega : u(x)=t_1\}$ is shown in Fig. 3.
\begin{center}
  \includegraphics[width=6cm,height=3.6cm]{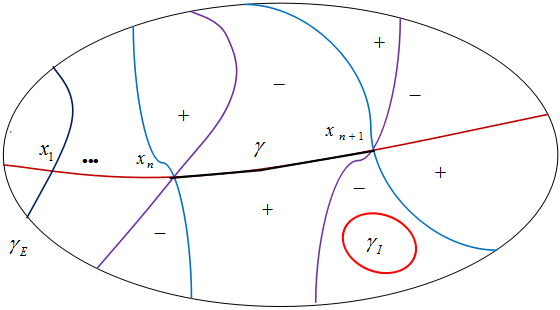}\\
  \scriptsize {\bf Fig. 3.}~~The distribution for the level lines of $\{x\in\Omega : u(x)=t_1\}$.
\end{center}
 Noting that both $A$ and $B$ are connected, by using Lemma \ref{le2.1} and the inductive assumption to $B,$ we know that $A$ contains exactly
 $$ \Big(\sum\limits_{i = 1}^n {{m_i}}+1\Big)+(m_{n+1}+1)-1=\sum\limits_{i = 1}^{n+1} {{m_i}}+1$$
connected components of the super-level set $\{x\in\Omega : u(x)>t_1\}$.

Step 2: When $q_1\geq 2.$  We know that the number of connected sets of the level lines $\{x\in \Omega : u(x)=t_1\}$ together with $x_1,\cdots,x_l$ increases one leading the number of connected components of $\{x\in\Omega : u(x)>t_1\}$ or $\{x\in\Omega : u(x)<t_1\}$ increases one. If the critical points $x_1,x_2,\cdots,x_l$ together with the level lines of $\{x\in \Omega : u(x)=t_1\}$ clustering round these points form $q_1$ connected sets. By the results of step 1, then we have
\begin{equation*}\begin{array}{l}
M_1\geq \sum\limits_{i = 1}^l {{m_i}}+1,~~~M_2\geq \sum\limits_{i = 1}^l {{m_i}}+1,
\end{array}\end{equation*}
and
\begin{equation*}\begin{array}{l}
M_1+M_2=2(\sum\limits_{i = 1}^l {{m_i}}+1)+(q_1-1)=2\sum\limits_{i = 1}^l {{m_i}}+q_1+1.
\end{array}\end{equation*}
This completes the proof of Case 2.

(iii) Case 3: We divide the proof of Case 3 into two steps. For convenience, up to renumbering, we denote $x_{l+1},\cdots,x_k$ and $m_{l+1},\cdots,m_k$ by $y_{1},\cdots,y_{k-l}$ and $\widetilde{m}_{1},\cdots,\widetilde{m}_{k-l}$ respectively.

Step 1: When $q_0=1,$ by induction. When $k-l=1,$ the result holds by Lemma \ref{le2.1} and Lemma \ref{le2.2}. For $1\leq k-l\leq n$, we assume that the connected set, which consists of $k-l$ critical points and the connected components clustering round these points, contains exactly $\sum_{i =1}^{k-l} {{\widetilde{m}_i}}$ components $\omega$ of the super-level set $\{x\in\Omega : u(x)>t_0\}$ such that $\omega$ meet the interior boundary $\gamma_I$. Let $k-l=n+1.$ Let $A$ be the set which consists of the points $y_1,\cdots,y_{n+1}$ together with the respective components clustering round these points. We may assume that the points $y_{1},\cdots,y_{n}$ together with the respective components clustering round these points form a connected set, denotes by $B.$ By Lemma \ref{le2.2} we know that $A$ cannot surround a component of $\{x\in\Omega : u(x)<t_0\}.$  Up to renumbering, therefore there is only one component of $\{x\in\Omega : u(x)>t_0\}$ whose boundary $\alpha$ contains both $y_{n}$ and $y_{n+1}$. The distribution for the level lines of $\{x\in\Omega : u(x)=t_0\}$ is shown in Fig. 4.
\begin{center}
  \includegraphics[width=6cm,height=3.6cm]{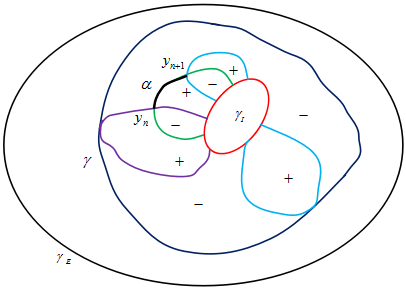}\\
  \scriptsize {\bf Fig. 4.}~~The distribution for the level lines of $\{x\in\Omega : u(x)=t_0\}$.
\end{center}
 Noting that both $A$ and $B$ are connected, by using Lemma \ref{le2.1} and the inductive assumption to $B,$ we know that $A$ contains exactly
 $$ \sum_{i = 1}^{n} {\widetilde{m}_i}+(\widetilde{m}_{n+1}+1)-1=\sum_{i = 1}^{n+1} {\widetilde{m}_i}=\sum\limits_{i = l+1}^{k} {{m_i}}$$
connected components $\omega$ of the super-level set $\{x\in\Omega : u(x)>t_0\}$ such that $\omega$ meet $\gamma_I$. This completes the proof of step 1.

Step 2: When $q_0\geq 2.$ Since the number of connected sets of the level lines $\{x\in\Omega : u(x)=t_0\}$ together with $x_{l+1},\cdots,x_k$ increases one leading the number of connected components of $\{x\in\Omega : u(x)>t_0\}$ increases one. If all the critical points $x_{l+1},\cdots,x_k$ together with the level lines $\{x\in\Omega : u(x)=t_0\}$ clustering round these points form $q_0$ connected sets. By the results of step 1, then we have
\begin{equation*}\begin{array}{l}
 \sharp\Big\{\mbox{the\ simply\ connected\ components\ $\omega$ of\ the\ super-level\ set} ~ \{x\in\Omega : u(x)>t_0\}\\~~~ \mbox{such\ that\ $\omega$ meet\ the\ interior\ boundary\ $\gamma_I$} \Big\}= \sum\limits_{i = l+1}^k {{m_i}}+(q_0-1).
\end{array}\end{equation*}
This completes the proof of Case 3.

(iv) Case 4: Case 2 implies Case 4.
\end{proof}

 \begin{Remark}\label{re2.6} Note that if $z_1<u(x_{l+1})=\cdots=u(x_k)=t_0<Z_1\leq z_2<u(x_1)=\cdots=u(x_l)=t_1<Z_2$. Suppose that there exists a non-simply connected component $\omega$ of $\{x\in \Omega : u(x)<t_1\}$ for critical value $t_1$, where the external boundary $\gamma_1$ of $\omega$ is a simply closed curve in $\Omega$ as in Case 1 of Lemma \ref{le2.5} or that there exists a simply closed curve $\gamma_0$ of $\{x\in \Omega : u(x)=t_0\}$ between $\gamma_I$ and $\gamma_E$ as in Case 3 of Lemma \ref{le2.5}, then $u$ has at least one critical point on $\gamma_1$ and $\gamma_0$ respectively. For the sake of clarity, we give the proof of the first situation. In fact, suppose by contradiction that $u$ has no critical point on $\gamma_1$. Without loss of generality, we may assume that $u$ has only two local maximal points $q_1, q_2$ on $\gamma_E$ and one critical point $x_1$ in $\Omega\setminus \overline{\omega}$ such that $u(x_1)=t_1$ and the multiplicity of $x_1$ is one, we denote the non-simply connected component of $\{x\in\Omega: u(x)>t_1\}$ by $A$. The distribution for the level lines of $\{x\in\Omega: u(x)=t_1\}$ is shown in Fig. 5.
 \begin{center}
  \includegraphics[width=6.0cm,height=4.1cm]{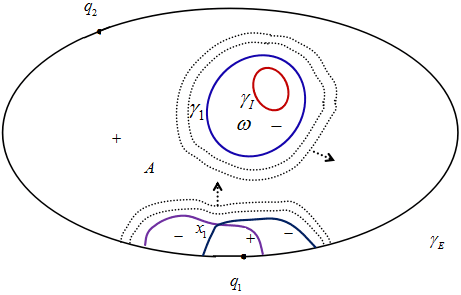}\\
  \scriptsize {\bf Fig. 5.}~~The distribution for the level lines of $\{x\in\Omega : u(x)=t_1\}$.
\end{center}
 By using the method of step 3 of the proof of Theorem 1.2 in \cite{Deng3}, this would imply that either: $u(x)=u(q_2)$ in interior points of $A$, or: there exist two level lines intersecting in $A$, i.e., there exist critical points in $A$. This is a contradiction.
\end{Remark}

\begin{Lemma}\label{le2.7}
Let $z_1<z_2<Z_1<Z_2$ and $x_1,x_2,\cdots,x_k$ be the interior critical points of $u$ in $\Omega$ and $m_1,m_2,\cdots,m_k$ be the multiplicities of critical points $x_1,x_2,\cdots,x_k$ respectively.\\
{\bf Situation 1}: Suppose that $z_1<u(x_{l+1})=\cdots=u(x_k)=t_0\leq z_2< Z_1\leq u(x_1)=\cdots=u(x_l)=t_1<Z_2$ or $z_2<u(x_{l+1})=\cdots=u(x_k)=t_0< Z_1\leq u(x_1)=\cdots=u(x_l)=t_1<Z_2$ or $z_1<u(x_{l+1})=\cdots=u(x_k)=t_0\leq z_2< u(x_1)=\cdots=u(x_l)=t_1<Z_1$ and that $x_1,\cdots,x_l$ and $x_{l+1},\cdots,x_k$ together with the corresponding level lines of $\{x\in\Omega : u(x)=t_1\}$ and $\{x\in\Omega : u(x)=t_0\}$ clustering round these points meet $\gamma_E,\gamma_I$ and form $q_1,q_0$ connected sets respectively, where $q_1,q_0\geq 1.$

Case 1: If there exists a simply closed curve $\gamma_1$ of $\{x\in \Omega : u(x)=t_1\}$ between $\gamma_I$ and $\gamma_E$ such that $u$ has at least one critical point on $\gamma_1$ and there exists a simply closed curve $\gamma_0$ of $\{x\in \Omega : u(x)=t_0\}$ between $\gamma_I$ and $\gamma_1$ such that $u$ has at least one critical point on $\gamma_0$ {\bf or} there exists totally one simply closed curve $\gamma$ of $\{x\in \Omega : u(x)=t\}$ for two critical values between $\gamma_I$ and $\gamma_E$ such that $u$ has at least one critical point on $\gamma.$ Then
 \begin{equation}\label{2.5}\begin{array}{l}
\sharp\Big\{\mbox{the\ simply\ connected\ components\ $\omega$ of\ the\ sub-level\ set} ~ \{x\in\Omega : u(x)<t_1\}\\~~~ \mbox{such\ that\ $\omega$ meet\ the\ external\ boundary\ $\gamma_E$} \Big\}\geq \sum\limits_{i = 1}^l {{m_i}}+q_1-1,
\end{array}\end{equation}
and
\begin{equation}\label{2.6}\begin{array}{l}
\sharp\Big\{\mbox{the\ simply\ connected\ components\ $\omega$ of\ the\ super-level\ set} ~ \{x\in\Omega : u(x)>t_0\}\\~~~ \mbox{such\ that\ $\omega$ meet\ the\ interior\ boundary\ $\gamma_I$} \Big\}\geq \sum\limits_{i = l+1}^k {{m_i}}+q_0-1.
\end{array}\end{equation}

Case 2: If there does not exist a simply closed curve $\gamma$ of $\{x\in \Omega : u(x)=t\}$ for any critical value $t$ between $\gamma_I$ and $\gamma_E$, then
 \begin{equation}\label{2.7}\begin{array}{l}
\sharp\Big\{\mbox{the\ simply\ connected\ components\ $\omega$ of\ the\ super-level\ set} ~ \{x\in\Omega : u(x)>t_1\}\\~~~ \mbox{such\ that\ $\omega$ meet\ the\ external\ boundary\ $\gamma_E$} \Big\}\geq\sum\limits_{i = 1}^l {{m_i}}+1,
\end{array}\end{equation}
and
 \begin{equation}\label{2.8}\begin{array}{l}
\sharp\Big\{\mbox{the\ simply\ connected\ components\ $\omega$ of\ the\ sub-level\ set} ~ \{x\in\Omega : u(x)<t_0\}\\~~~ \mbox{such\ that\ $\omega$ meet\ the\ external\ boundary\ $\gamma_I$} \Big\}\geq\sum\limits_{i = l+1}^k {{m_i}}+1.
\end{array}\end{equation}
{\bf Situation 2}: Suppose that $z_2<u(x_1)=u(x_2)=\cdots=u(x_k)=t<Z_1$ and that $x_1,\cdots,x_k$ together with the corresponding level lines of $\{x\in\Omega : u(x)=t\}$ clustering round these points form $q$ connected sets, where $q\geq 1.$ In addition, we set $M_1$ and $M_2$ as the number of the connected components of $\{x\in\Omega : u(x)>t\}$ and $\{x\in\Omega : u(x)<t\}$, respectively.

Case 3:  If there exists a simply closed curve $\gamma$ of $\{x\in \Omega : u(x)=t\}$ between $\gamma_I$ and $\gamma_E$ such that $u$ has at least two critical points on $\gamma$, then
 \begin{equation}\label{2.9}\begin{array}{l}
M_1\geq \sum\limits_{i = 1}^k {{m_i}},~M_2\geq \sum\limits_{i = 1}^k {{m_i}}, ~\mbox{and}~M_1+M_2=2\sum\limits_{i = 1}^k {{m_i}}+q-1,
\end{array}\end{equation}

Case 4: If there does not exist a simply closed curve $\gamma$ of $\{x\in \Omega : u(x)=t\}$ for critical value $t$ between $\gamma_I$ and $\gamma_E$, then
\begin{equation}\label{2.10}\begin{array}{l}
M_1\geq \sum\limits_{i = 1}^k {{m_i}}+1,~M_2\geq \sum\limits_{i = 1}^k {{m_i}}+1, ~\mbox{and}~M_1+M_2=2\sum\limits_{i = 1}^k {{m_i}}+q+1.
\end{array}\end{equation}
\end{Lemma}
\begin{proof}[Proof] According to the fine analysis about the distributions of connected components of the super-level sets $\{x\in \Omega: u(x)>t\}$ and sub-level sets $\{x\in \Omega: u(x)<t\}$ for some $t$, we can easily know that the Case 1, 2, 3, 4 of Lemma \ref{le2.5} implies Case 1, 2, 3, 4 respectively.
\end{proof}

\section{Proof of Theorem \ref{th1.1}}
 ~~~~In this section, we only give out the proofs of these two cases $z_1<Z_1\leq z_2<Z_2$ and $z_1<z_2<Z_1 <Z_2$. It can be easily observed from the proofs that other cases can be dealt with in a similar way.

\subsection{Proof for the case of $z_1<Z_1\leq z_2<Z_2$}

 ~~~~~In this subsection, we suppose that $x_1,x_2,\cdots,x_k$ are the interior critical points of $u$ in $\Omega$ and $z_1<u(x_{l+1}),\cdots,u(x_k)<Z_1\leq z_2<u(x_1),\cdots,u(x_l)<Z_2$. From the preparations in Section 2, we are now ready to present the proof of Theorem \ref{th1.1} for the case of $z_1<Z_1\leq z_2<Z_2$.
\begin{proof}[Proof of Theorem \ref{th1.1} for the case of $z_1<Z_1\leq z_2<Z_2$]  (i) Case $A_1$: If $z_1<u(x_{l+1})=\cdots=u(x_k)=t_0<Z_1\leq z_2<u(x_1)=\cdots=u(x_l)=t_1<Z_2$, by the results of Lemma \ref{le2.5}, we know that
\begin{equation*}\begin{array}{l}
 \sharp\Big\{\mbox{the\ simply\ connected\ components\ $\omega$ of\ the\ sub-level\ set} ~ \{x\in\Omega : u(x)<t_1\}\\~~~ \mbox{such\ that\ $\omega$ meet\ the\ external\ boundary\ $\gamma_E$} \Big\}\geq \sum\limits_{i = 1}^l {{m_i}},
\end{array}\end{equation*}
and
\begin{equation*}\begin{array}{l}
\sharp\Big\{\mbox{the\ simply\ connected\ components\ $\omega$ of\ the\ super-level\ set} ~ \{x\in\Omega : u(x)>t_0\}\\~~~ \mbox{such\ that\ $\omega$ meet\ the\ interior\ boundary\ $\gamma_I$} \Big\}\geq \sum\limits_{i = l+1}^k {{m_i}}.
\end{array}\end{equation*}
By the strong maximum principle, we have that $u$ exists at least $\sum\limits_{i =1}^l {{m_i}}$ and $\sum\limits_{i =l+1}^k {{m_i}}$ local minimal points and maximal points on $\gamma_E$ and $\gamma_I$ respectively.
Hence, we have
\begin{equation*}\begin{array}{l}
\sum\limits_{i =1}^l {{m_i}}\leq N_2, ~~~\sum\limits_{i = l+1}^k {{m_i}}\leq N_1,
 \end{array}\end{equation*}
then
\begin{equation*}\begin{array}{l}
 \sum\limits_{i = 1}^k {{m_i}}\leq N_1+ N_2.
 \end{array}\end{equation*}

(ii) Case $A_2$: If the values at critical points $x_1,\cdots,x_l$ are not totally equal or the values at critical points $x_{l+1},\cdots,x_k$ are not totally equal. Next we need divide the proof of Case $A_2$ into two situations.

(1) Situation 1: If the values at critical points $x_{1},\cdots,x_l$ are not totally equal, without loss of generality, we may assume that
\begin{equation}\label{3.1}
\begin{split}
z_2<u(x_1)&=\cdots=u(x_{j_1})<u(x_{j_1+1})=\cdots=u(x_{j_2})<\cdots<\cdots\\
&<u(x_{j_{n-1}+1})=\cdots=u(x_{j_{n}}),
\end{split}
\end{equation}
where $x_1,\cdots,x_{j_1},\cdots, x_{j_2},\cdots, x_{j_{n}}$ are different critical points in $\Omega$, $j_{n}=l$ and $n\geq 2.$ Next we put
\begin{equation*}\begin{array}{l}
E_{j}:=\Big\{\omega : \mbox{open\ set}~ \omega ~ \mbox{is\ a\ connected\ component\ of} ~ \{x\in \Omega: u(x)>u(x_j)\}\Big\}~(j=j_1,j_2,\cdots,j_n),
\end{array}\end{equation*}
and
\begin{equation*}\begin{array}{l}
G_{j_n}:=\Big\{\omega : \mbox{open\ set}~ \omega~ \mbox{is\ a\ connected\ component\ of} ~ \{x\in \Omega: u(x)>u(x_{j})\}(j=j_1,j_2,\cdots,j_n)\\~~~~~~~~~~~~ \mbox{such\ that\ there\ does\ not\ exist\ interior\ critical\ point\ in} ~\omega\Big\}.
\end{array}\end{equation*}
By the definition, we know that $G_{j_n}$ consists of disjoint connected components. We set
\[|G_{j_n}|:=\sharp\big\{\omega : \omega~ \mbox{is\ a\ connected\ component\ of} ~G_{j_n}\big\}.\]

To illustrate $G_{j_n}$, let us consider an illustration for $G_{j_2}$. Assume that $u$ has only three critical points $x_{j_1},x_{j_1+1},x_{j_2}$ such that $z_2<u(x_{j_1})<u(x_{j_1+1})=u(x_{j_2})$ in $\Omega$ and the respective multiplicity $m_{j_1}=1,m_{j_1+1}=1,m_{j_2}=2$. The distributions of elements of $G_{j_2}$ is shown in Fig. 6.
\begin{center}
  \includegraphics[width=6.6cm,height=4.0cm]{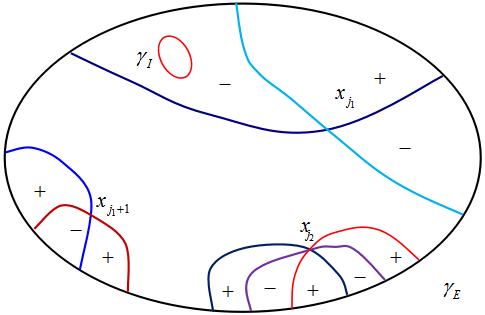}\\
  \scriptsize {\bf Fig. 6.}~~The distributions of elements of $G_{j_2}$.
\end{center}
By the definition of $G_{j_n},$ we know $|G_{j_2}|=6.$

Now let us show that $\big|G_{j_s}\big|\geq \sum_{i = 1}^{{j_s}} m_i+1$ by induction on the number $s.$ When $s=1,$ the result holds by Lemma \ref{le2.2}. Suppose that $\big|G_{j_s}\big|\geq \sum_{i = 1}^{{j_s}} m_i+1$ for $1\leq s\leq n-1.$  Let $s=n.$ Then, by (\ref{3.1}) and the definition of $E_{j}$, we have
\[\{x_{j_{n-1}+1},\cdots,x_{j_{n}}\}\subset \bigcup \limits_{\omega  \in {E_{{j_{n-1}}}}} \omega,\]
where $\omega$ is a connected component of $\{x\in \Omega: u(x)>u(x_{j_{n-1}})\}.$

Let us assume that $\{x_{j_{n-1}+1},\cdots,x_{j_{n}}\}$ are contained in exactly $\widehat{q}$ components $\omega_1,\cdots,\omega_{\widehat{q}}.$ Then $x_{j_{n-1}+1},\cdots,x_{j_{n}}$ together with the corresponding level lines of $\{x\in\Omega: u(x)=u(x_{j_n})\}$ clustering round these points at least form $\widehat{q}$ connected sets. By the Case 2 of Lemma \ref{le2.5}, we have
\begin{equation*}\begin{array}{l}
M:=\sharp\Big\{\mbox{the\ connected\ components\ of}~ \{x\in \Omega: u(x)>u(x_{j_{n}})\}~\mbox{in\ all}~\omega_j~(j=1,\cdots,\widehat{q})\Big\}\\ ~~~~\geq \sum\limits_{i=j_{n-1}+ 1}^{j_{n}} {{m_i}}+\widehat{q}.
\end{array}\end{equation*}
By using the definition of $\big|G_{j_{n}}\big|$ and the inductive assumption to $1\leq s\leq n-1,$ since $\{x_{j_{n-1}+1},\cdots,x_{j_{n}}\}$ are contained in exactly $\widehat{q}$ components $\omega_1,\cdots,\omega_{\widehat{q}},$ so when we calculate the number of $|G_{j_{n}}|$, the number of $|G_{j_{n-1}}|$ will be reduced by $\widehat{q}.$ Then we have
\begin{equation*}\begin{array}{l}
\big|G_{j_{n}}\big|=\big|G_{j_{n-1}}\big|+M-\widehat{q}\geq\big|G_{j_{n-1}}\big|+\Big(\sum\limits_{i=j_{n-1}+ 1}^{j_{n}} {{m_i}}+\widehat{q}\Big)-\widehat{q}\geq\sum\limits_{i=1}^{j_{n}} {{m_i}}+1.
\end{array}\end{equation*}
By the strong maximum principle and Lemma \ref{le2.2}, we have that $u$ has at least $\sum_{i = 1}^l {{m_i}}+1$ local maximum points on $\gamma_E.$ Therefore, we obtain
\begin{equation*}\begin{array}{l}
\sum\limits_{i = 1}^l {{m_i}}+1\leq N_2.
\end{array}\end{equation*}

(2) Situation 2: If the values at critical points $x_{l+1},\cdots,x_k$ are not totally equal. For convenience, we denote $x_{l+1},\cdots,x_k$ and $m_{l+1},\cdots,m_k$ by $y_{1},\cdots,y_{k-l}$ and $\widetilde{m}_{1},\cdots,\widetilde{m}_{k-l}$ respectively. The proof is same as the case 2 of the proof of Theorem 1.1 in \cite{Deng3}. Without loss of generality, we may assume that
\begin{equation}\label{3.2}
\begin{split}
u(y_1)&=\cdots=u(y_{j_1})<u(y_{j_1+1})=\cdots=u(y_{j_2})<\cdots<\cdots\\
&<u(y_{j_{n-1}+1})=\cdots=u(y_{j_{n}})<Z_1,
\end{split}
\end{equation}
where $y_1,\cdots,y_{j_1},\cdots, y_{j_2},\cdots, y_{j_{n}}$ are different critical points in $\Omega$, $j_{n}=k-l$ and $n\geq 2.$ Next we put
\begin{equation*}\begin{array}{l}
E_{j}:=\Big\{\omega : \mbox{open\ set}~ \omega ~ \mbox{is\ a\ connected\ component\ of} ~ \{x\in \Omega: u(x)>u(y_j)\}\Big\}~(j=j_1,j_2,\cdots,j_n),
\end{array}\end{equation*}
and
\begin{equation*}\begin{array}{l}
F_{j_n}:=\Big\{\omega : \mbox{open\ set}~ \omega~ \mbox{is\ a\ connected\ component\ of} ~ \{x\in \Omega: u(x)<u(y_{j_1})\}~\mbox{or} ~\omega~ {is\ a} \\ ~~~~~~~~~~~~ \mbox{connected\ component\ of} ~ \{x\in \Omega: u(y_{j_i})<u(x)<u(y_{j_{i+1}})~\mbox{for\ some}~1\leq i\leq n-1 \}\Big\}.
\end{array}\end{equation*}
By the definition, we know that $F_{j_n}$ consists of disjoint connected components. We set
\[|F_{j_n}|:=\sharp\big\{\omega : \omega~ \mbox{is\ a\ connected\ component\ of} ~F_{j_n}\big\}.\]

To illustrate $F_{j_n}$, let us consider an illustration for $F_{j_2}$. Assume that $u$ has only two critical points $y_{j_1},y_{j_2}$ such that $u(y_{j_1})<u(y_{j_2})<Z_1$ in $\Omega$ and the respective multiplicity $\widetilde{m}_{j_1}=1,\widetilde{m}_{j_2}=1$. The distributions of elements of $F_{j_2}$ is shown in Fig. 7.
\begin{center}
  \includegraphics[width=6.6cm,height=3.9cm]{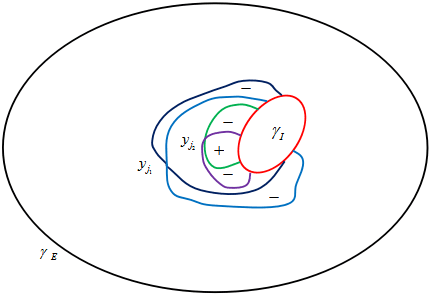}\\
  \scriptsize {\bf Fig. 7.}~~The distributions of elements of $F_{j_2}$.
\end{center}
By the definition of $F_{j_n},$ we know $|F_{j_2}|=4.$

Now let us show that $\big|F_{j_s}\big|\geq \sum_{i = 1}^{{j_s}} \widetilde{m}_i+1$ by induction on the number $s.$ When $s=1,$ the result holds by Lemma \ref{le2.2}. Suppose that $\big|F_{j_s}\big|\geq \sum_{i = 1}^{{j_s}} \widetilde{m}_i+1$ for $1\leq s\leq n-1.$  Let $s=n.$ Then, by (\ref{3.2}) and the definition of $E_{j}$, we have
\[\{y_{j_{n-1}+1},\cdots,y_{j_{n}}\}\subset \bigcup \limits_{\omega  \in {E_{{j_{n-1}}}}} \omega,\]
where $\omega$ is a connected component of $\{x\in \Omega: u(x)>u(y_{j_{n-1}})\}.$

Let us assume that $\{y_{j_{n-1}+1},\cdots,y_{j_{n}}\}$ are contained in exactly $\widetilde{q}$ components $\omega_1,\cdots,\omega_{\widetilde{q}}.$ Then $\{y_{j_{n-1}+1},\cdots,y_{j_{n}}\}$ together with the corresponding level lines of $\{x\in \Omega: u(x)=u(y_{j_{n}})\}$ clustering round these points at least form $\widetilde{q}$ connected sets. By Lemma \ref{le2.2}, we have
\begin{equation*}\begin{array}{l}
\widetilde{M}:=\sharp\Big\{\mbox{the\ connected\ components\ of}~ \{x\in \Omega: u(y_{j_{n-1}})<u(x)<u(y_{j_{n}})\}~\mbox{in\ all}~\omega_j~(j=1,\cdots, \widetilde{q})\Big\}\\ ~~~~\geq \sum\limits_{i=j_{n-1}+ 1}^{j_{n}} {\widetilde{m}_i}+\widetilde{q}.
\end{array}\end{equation*}
By using the definition of $\big|F_{j_{n}}\big|$ and the inductive assumption to $1\leq s\leq n-1,$ then we have
\begin{equation*}\begin{array}{l}
\big|F_{j_{n}}\big|=\big|F_{j_{n-1}}\big|+\widetilde{M} \geq \big|F_{j_{n-1}}\big|+\Big(\sum\limits_{i=j_{n-1}+ 1}^{j_{n}} {\widetilde{m}_i}+\widetilde{q}\Big)-\widetilde{q} \geq\sum\limits_{i=1}^{j_{n}} {{\widetilde{m}_i}}+1.
\end{array}\end{equation*}
By the strong maximum principle, we have that $u$ has at least $\sum_{i = 1}^{k-l} {{\widetilde{m}_i}}+1$ local minimal points on $\gamma_I.$ Therefore, we obtain
\begin{equation*}\begin{array}{l}
\sum\limits_{i = 1}^{k-l} {{\widetilde{m}_i}}+1=\sum\limits_{i = l+1}^{k} {{m_i}}+1\leq N_1.
\end{array}\end{equation*}

According to Case $A_1$, Case $A_2$ and the assumption of
\begin{equation*}\label{3.3}
\begin{split}
z_1<u(x_{l+1}),\cdots,u(x_k)<Z_1\leq z_2<u(x_1),\cdots,u(x_l)<Z_2,
\end{split}
\end{equation*}
then we have
 \[\sum\limits_{i = l+1}^{k} {{m_i}}\leq N_1,~~~\sum\limits_{i=1}^{l} {{m_i}}\leq N_2,\]
that is
\begin{equation}\label{3.3}
\begin{split}
\sum\limits_{i = 1}^{k}{{m_i}}\leq N_1+N_2.
\end{split}
\end{equation}
This completes the proof of Theorem \ref{th1.1} for the case of $z_1<Z_1\leq z_2<Z_2$.
\end{proof}

\subsection{Proof for the case of $z_1<z_2<Z_1<Z_2$}
 ~~~~~In this subsection, we suppose that $x_1,x_2,\cdots,x_k$ are the interior critical points of $u$ in $\Omega$ and $z_1<z_2<Z_1<Z_2$. From the preparations in Section 2, we are now ready to present the proof of Theorem \ref{th1.1} for the case of $z_1<z_2<Z_1<Z_2$.

\begin{proof}[Proof of Theorem \ref{th1.1} for the case of $z_1<z_2<Z_1<Z_2$] (i) Case $B_1$: If $z_1<u(x_{l+1})=\cdots=u(x_k)=t_0\leq z_2<Z_1\leq u(x_1)=\cdots=u(x_l)=t_1<Z_2$ or  $z_2<u(x_{l+1})=\cdots=u(x_k)=t_0<Z_1\leq u(x_1)=\cdots=u(x_l)=t_1<Z_2$ or $z_1<u(x_{l+1})=\cdots=u(x_k)=t_0\leq z_2< u(x_1)=\cdots=u(x_l)=t_1<Z_1$, by the Case 1 and Case 2 of Lemma \ref{le2.7} and the strong maximum principle, we know that
\begin{equation*}\begin{array}{l}
\sum\limits_{i =1}^l {{m_i}}\leq N_2, ~~~\sum\limits_{i = l+1}^k {{m_i}}\leq N_1,
 \end{array}\end{equation*}
then
\begin{equation*}\begin{array}{l}
 \sum\limits_{i = 1}^k {{m_i}}\leq N_1+ N_2.
 \end{array}\end{equation*}

 (ii) Case $B_2$: If $ z_2<u(x_1)=u(x_2)=\cdots=u(x_k)=t<Z_1$. By the Case 3 and Case 4 of Lemma \ref{le2.7} and the strong maximum principle, we have
\begin{equation*}\begin{array}{l}
\sum\limits_{i =1}^k {{m_i}}\leq N_1+N_2.
 \end{array}\end{equation*}

(iii) Case $B_3$: For the other cases, the idea of proof is essentially same as the Case $A_2$ in section 3.1. Here we omit the proof.
\end{proof}

\section{Proof of Theorem \ref{th1.2}}
~~~~~In this section, we assume that $z_2\geq Z_1.$ We investigate the geometric structure of interior critical point sets of a solution $u$ in a planar bounded smooth multiply connected domain $\Omega$ with one interior boundary $\gamma_I$ and the external boundary $\gamma_E$, where $u$ has only $N_1$ ($N_2$) equal local maxima and $N_1$ ($N_2$) equal local minima relative to $\overline{\Omega}$ on $\gamma_I$ ($\gamma_E$), i.e., the values of all local maximal (minimum) points on the corresponding boundary are equal. We develop a new method to prove that one of the following three results holds $\sum_{i = 1}^k {{m_i}}= N_1+N_2$ or $\sum_{i = 1}^k {{m_i}}+1= N_1+N_2$ or $\sum_{i = 1}^k {{m_i}}+2= N_1+N_2$, where $N_1,N_2\geq 2$.
\begin{proof}[Proof of Theorem \ref{th1.2}] We need divide the proof into three cases.

Firstly, we should show that $u$ has at least one interior critical point in $\Omega.$ Suppose by contradiction that $|\nabla u|>0$ in $\Omega$. The strong maximum principle implies that $u$ has no interior maximum point and minimal point in $\Omega,$ then we have
 $$z_1<u(x)<Z_2 ~\mbox{for\ any}~ x\in \Omega.$$
According to the assumption of Theorem \ref{th1.2} and Remark \ref{re1.5}, let $q_1,\cdots,q_{N_2}$ and $p_1,\cdots,p_{N_2}$ be the local maximal points and local minimal points relative to $\overline{\Omega}$ on $\gamma_E$ respectively and $u(q_1)=\cdots=u(q_{N_2}),$ $u(p_1)=\cdots=u(p_{N_2})$.

Without loss of generality, we may assume that there only exist two different local maximal points $q_1,q_2$ and minimum points $p_1,p_2$ on boundary $\gamma_E.$ Note that $u$ is monotonically decreasing on the connected components of $\gamma_E$ from one maximal point to the near minimum point. Therefore, $\{x\in \Omega: u(x)=Z_2-\epsilon\}$ exactly exists two level lines in $\Omega$ for any $\epsilon$ such that $0<\epsilon<Z_2-z_2$. Here, we need consider the following two situations:

 Situation 1: If $Z_1=z_2$, this is impossible, because this would imply that either: $u(x)=\mathop {\min}\limits_{\gamma_E}\psi_2(x)=z_2~\mbox{in\ interior\ points\ of}~\Omega$, i.e., there exist two level lines of $\{x\in \Omega;u(x)=z_2\}$ connecting $\gamma_I$ from $p_1,p_2$ respectively, this contradicts with the continuity of $u$, or: there exist two level lines of $\{x\in \Omega;u(x)=t_0\in (z_2,Z_2)\}$ intersecting in $\Omega$, i.e., there exist critical points in $\Omega$ (see Fig. 8).
\begin{center}
  \includegraphics[width=12cm,height=3.2cm]{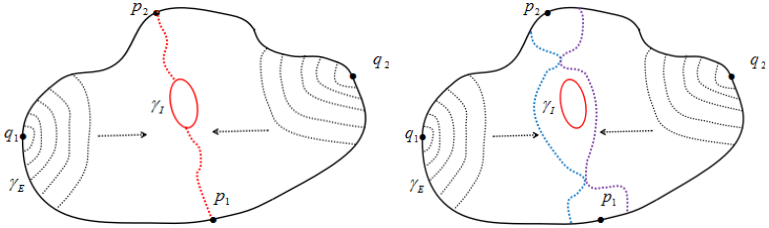}\\
  \scriptsize {\bf Fig. 8.}~~The distributions of some level lines $\{x\in A: u(x)=t\in (z_2,Z_2)\}.$
\end{center}

 Situation 2: If $Z_1<z_2$, this is also impossible. Because this would imply that there exist two level lines of $\{x\in \Omega;u(x)=t_0\in (z_2,Z_2)\}$ intersecting in $\Omega$, i.e., there exist critical points in $\Omega.$ Then $u$ has at least one interior critical point in $\Omega$.

(i) Case $C_1$: Suppose that there exists a non-simply connected component $\omega$ of $\{x\in \Omega : u(x)<t\}$ for some $t\in (z_2,Z_2)$ such that $\omega$ meets $\gamma_E$ and that the Case 3 of Lemma \ref{le2.5} does not occur for $t\in (z_1,Z_1)$. Next we need divide the proof of Case $C_1$ into four steps.

Step 1, the first ``{just right}'': we show that all critical values are equal to one of the two values. By Lemma \ref{le2.3}, we suppose that the interior critical points of $u$ are $x_1,x_2,\cdots,x_k$ and $z_1<u(x_{l+1}),\cdots,u(x_k)<Z_1\leq z_2<u(x_1),\cdots,u(x_l)<Z_2$. Next we show that $z_1<u(x_{l+1})=\cdots=u(x_k)<Z_1\leq z_2<u(x_{1})=\cdots=u(x_l)<Z_2.$ Here, we only give the proof of $u(x_{1})=\cdots=u(x_l),$ the proof of $u(x_{l+1})=\cdots=u(x_k)$ is similar.  By the assumption of Theorem \ref{th1.2}, let $q_1,\cdots,q_{N_2}$ and $p_1,\cdots,p_{N_2}$ be the equal local maximal points and minimum points on $\gamma_E,$ respectively. We set up the usual contradiction argument. Without loss of generality, we suppose that $z_2<u(x_1)<u(x_2)<Z_2$ and that $m_1,m_2$ are the respective multiplicities of $x_1,x_2.$ Then, by Lemma \ref{le2.2}, we know that any connected component of $\{x\in \Omega:u(x)>u(x_1)\}$ has to meet the boundary $\gamma_E$ and $u(p_1)<u(x_1)$, where $p_1$ is the minimal point of some one connected component of $\{x\in \Omega:u(x)<u(x_1)\}$ on $\gamma_E$. At the same time, there exists one connected component $C$ of $\{x\in \Omega:u(x_1)<u(x)<u(x_2)\}$ meeting $\gamma_E$ such that $u(x_1)<u(p_2)<u(x_2)$, where $p_2$ is the minimal point of connected component $C$ on $\gamma_E$ (see Fig. 9). Then $u(p_1)\neq u(p_2)$, which contradicts with the assumption of Theorem \ref{th1.2}. This completes the proof of step 1.
\begin{center}
  \includegraphics[width=5.5cm,height=3.7cm]{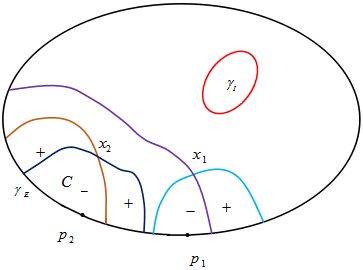}\\
  \scriptsize {\bf Fig. 9.} ~~The distributions of the connected components.
\end{center}

Step 2, the second ``{just right}'': we show that critical points together with the corresponding level lines of $\{x\in \Omega: u(x)=t~\mbox{for some}~t\in (z_2,Z_2)\}$ ($\{x\in \Omega: u(x)=t~\mbox{for some}~t\in (z_1,Z_1)\}$), which meet $\gamma_E$ ($\gamma_I$), clustering round these points exactly form one connected set. Without loss of generality, we suppose by contradiction that $x_1,x_2,\cdots,x_{l}$ together with the level lines of $\{x\in \Omega: u(x)=t>z_2\}$ clustering round these points  form two connected sets.
Therefore, there exists a connected components of $\{x\in \Omega: u(x)<t\}$, which meets two parts $\gamma_1,\gamma_2$ of $\gamma_E$, denoting by $A$ (see Fig. 10).

Note that $u$ is monotonically decreasing on the connected components of boundary $\gamma_E$ from one maximal point to the near minimum point. Therefore, $\{x\in A: u(x)=t-\epsilon\}$ exactly exists two level lines in $A$ for some $\epsilon$ such that $0<\epsilon<t-z_2$. Here, we need consider the following two situations:

 Situation 1: If $Z_1=z_2$, this is impossible, because this would imply that either: $u(x)=\mathop {\min}\limits_{\gamma_E}\psi_2(x)=z_2~\mbox{in\ interior\ points\ of}~A$, i.e., there exist two level lines of $\{x\in A;u(x)=z_2\}$ connecting $\gamma_I$ from $p_1,p_2$ respectively, this contradicts with the continuity of $u$, or: there exist two level lines of $\{x\in A;u(x)=t_0\in (z_2,t)\}$ intersecting in $A$, i.e., there exist critical points in $A$ (see Fig. 10).
\begin{center}
  \includegraphics[width=11cm,height=3.9cm]{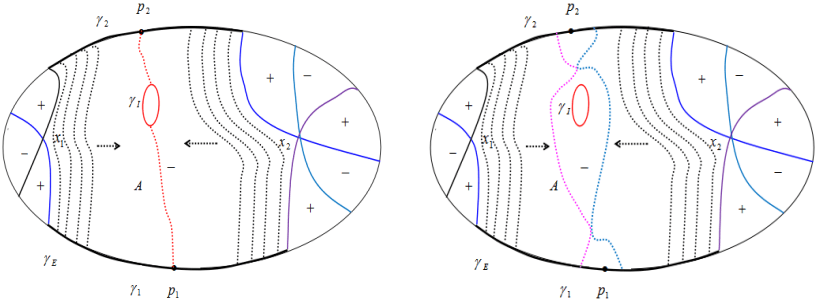}\\
  \scriptsize {\bf Fig. 10.}~~The distributions of some level lines $\{x\in A: u(x)=t-\epsilon, 0<\epsilon<t-z_2\}.$
\end{center}

 Situation 2: If $Z_1<z_2$, this is also impossible. Because this would imply that there exist two level lines of $\{x\in A;u(x)=t_0\in (z_2,t)\}$ intersecting in $A$, i.e., there exist critical points in $A.$ This completes the proof of step 2.

Step 3, the third ``{just right}'': we show that every connected component of $\{x\in \Omega: u(x)>t>z_2\}$ ($\{x\in \Omega: u(x)<t~\mbox{for}~t\in (z_1,Z_1)\}$) has exactly one global maximal (minimum) point on boundary $\gamma_E$ ($\gamma_I$). In fact, we assume that some connected component $B$ of $\{x\in \Omega: u(x)>t>z_2\}$ exists two global maximal points on boundary $\gamma_E.$ According to $u$ has $N_2$ equal local maxima and $N_2$ equal local minima on $\gamma_E$, then there must exist a minimum point $\widetilde{p}$ between the two maximal points on $\gamma_E$ such that $u(\widetilde{p})=z_2$. Since $u(x)>t>z_2$ in $B$, by the continuity of solution $u$, this contradicts with the definition of connected component $B$. This completes the proof of step 3.

Step 4, By the results of step 2 and the results of Case 2 and Case 4 in Lemma \ref{le2.5}, we have
 \begin{equation*}\begin{array}{l}
 \sharp\Big\{\mbox{the\ connected\ components\ of\ the\ super-level\ set} ~ \{x\in\Omega : u(x)>t~\mbox{for some}~t\in (z_2,Z_2)\} \Big\}\\
  ~~= \sum\limits_{i=1}^l {{m_i}}+1,
\end{array}\end{equation*}
and
\begin{equation*}\begin{array}{l}
\sharp\Big\{\mbox{the\ connected\ components\ of\ the\ sub-level\ set} ~ \{x\in\Omega : u(x)<t~\mbox{for some}~t\in (z_1,Z_1)\} \Big\}\\
~~= \sum\limits_{i = {l+1}}^{k} {{m_i}}+1.
\end{array}\end{equation*}
On the other hand, by the results of step 3 and the strong maximum principle, therefore we obtain
\begin{equation*}\begin{array}{l}
 \sum\limits_{i = 1}^l {{m_i}}+1=N_2 ~~\mbox{and}~~\sum\limits_{i = l+1}^{k} {{m_i}}+1=N_1.
\end{array}\end{equation*}
That is
\begin{equation}\label{4.1}\begin{array}{l}
\sum\limits_{i = 1}^{k}{{m_i}}+2= N_1+N_2.
\end{array}\end{equation}

(ii) Case $C_2$: Suppose that there exists a non-simply connected component $\omega$ of $\{x\in \Omega : u(x)<t\}$ for some $t\in (z_2,Z_2)$ and the external boundary $\gamma$ of $\omega$ is a simply closed curve between $\gamma_I$ and $\gamma_E$ such that $u$ has at least one critical point on $\gamma$. The idea of proof is essentially same as the proof of Case $C_1$. Next we need divide the proof of Case $C_2$ into four steps.

Step 1, the first ``just right'': According to Lemma \ref{le2.3}, we assume that $z_2<u(x_1),\cdots,u(x_l)<Z_2$. We show that $u(x_1)=\cdots=u(x_l)=t,$ i.e., we exclude the situation 1 of case $A_2$ in Theorem \ref{th1.1}. The proof is same as the step 1 of the proof of Case $C_1$.

Step 2, the second ``just right'': we show that $x_1,x_2,\cdots,x_l$ together with the corresponding level lines of $\{x\in \Omega: u(x)=t\}$ clustering round these points exactly form one connected set. The proof is same as the step 2 of the proof of Case $C_1$.

Step 3, the third ``just right'': we show that every simply connected component $\omega$ of $\{x\in \Omega: u(x)<t\}$ has exactly one minimum point on $\gamma_E,$ where $\omega$ meets the external boundary $\gamma_E.$  In fact, we assume that some simply connected component $\omega$ of $\{x\in \Omega: u(x)<t\}$ exists two minimum points on boundary $\gamma_E.$ According to $u$ has only $N_2$ local maximal points and $N_2$ local minimum points on $\gamma_E$, then there must exist a maximal point $\overline{q}$ between the two minimum points on $\gamma_E$ such that $u(\overline{q})=Z_2$. Since $u(x)<t<Z_2$ in $\omega$, by the continuity of solution $u$, this contradicts with the definition of connected component $\omega$. This completes the proof of step 3.

Step 4, By the results of step 2 and the results of step 1 of Case 1 in Lemma \ref{le2.5}, we have
 \begin{equation*}\begin{array}{l}
 \sharp\Big\{\mbox{the\ simply\ connected\ components\ $\omega$ of\ the\ sub-level\ set} ~ \{x\in\Omega : u(x)<t\}\\~~~ \mbox{such\ that\ $\omega$ meet\ the\ external\ boundary\ $\gamma_E$} \Big\}=\sum\limits_{i = 1}^l {{m_i}}.
\end{array}\end{equation*}
On the other hand, using the results of step 3 and the strong maximum principle, therefore we obtain
\begin{equation*}\begin{array}{l}
 \sum\limits_{i = 1}^l {{m_i}}=N_2.
\end{array}\end{equation*}

(iii) Case $C_3$: Suppose that there exists a simply closed curve $\gamma$ of $\{x\in \Omega : u(x)=t\}$ for $t\in (z_1,Z_1)$ between $\gamma_I$ and $\gamma_E$ such that $u$ has at least one critical point on $\gamma$ and that $z_1<u(x_{l+1}),\cdots,u(x_k)<Z_1.$ The proof is same as the proof of Case $C_2$. Then we have
\begin{equation*}\begin{array}{l}
 \sum\limits_{i = {l+1}}^k {{m_i}}=N_1.
\end{array}\end{equation*}

According to the above discussion, if Case $C_2$ and Case $C_3$ both occur, then we have
\begin{equation}\label{4.2}\begin{array}{l}
 \sum\limits_{i = {1}}^k {{m_i}}=N_1+N_2.
\end{array}\end{equation}

In addition, if Case 1 and Case 4 of Lemma \ref{le2.5} occur, or Case 2 and Case 3 of Lemma \ref{le2.5} occur, then we have
\begin{equation}\label{4.3}\begin{array}{l}
 \sum\limits_{i = {1}}^k {{m_i}}+1=N_1+N_2.
\end{array}\end{equation}
This completes the proof of Theorem \ref{th1.2}.
\end{proof}

In addition, according to Theorem \ref{th1.2} and Lemma \ref{le2.7}, we can easily have the following results.
\begin{Corollary}\label{co4.1}
  Suppose that domain $\Omega$ satisfies the hypothesis of Theorem \ref{th1.1}, $z_1<z_2<Z_1<Z_2$ and that $u$ is a non-constant solution of (\ref{1.2}). In addition, suppose that $u$ has only $N_1$ equal local maxima and $N_1$ equal local minima relative to $\overline{\Omega}$ on $\gamma_I$ and that $u$ has only $N_2$ equal local maxima and $N_2$ equal local minima relative to $\overline{\Omega}$ on $\gamma_E$. Then $u$ has finite interior critical points, and one of the following three holds
\begin{equation}\label{4.4}\begin{array}{l}
\begin{array}{l}
\sum\limits_{i = 1}^k {{m_i}}= N_1+N_2,
\end{array}
\end{array}\end{equation}
or
\begin{equation}\label{4.5}\begin{array}{l}
\begin{array}{l}
\sum\limits_{i = 1}^k {{m_i}}+1= N_1+N_2,
\end{array}
\end{array}\end{equation}
or
\begin{equation}\label{4.6}\begin{array}{l}
\begin{array}{l}
\sum\limits_{i = 1}^k {{m_i}}+2= N_1+N_2,
\end{array}
\end{array}\end{equation}
where $m_i$ is as in Theorem \ref{th1.1}.
\end{Corollary}

\begin{Remark}\label{re4.2}
 Notice that in Theorem \ref{th1.2} and Corollary \ref{co4.1}, the assumption of all the local maximal and local minimum points of $u$ on $\gamma_I$ and $\gamma_E$ are the local maximal and local minimum points relative to $\overline{\Omega}$ is necessary. If there exist some local maximal and local minimum points only relative to $\gamma_I$ and $\gamma_E$, there will be two counterexamples.

 \vspace{0.2cm}\noindent {\bf Counterexample 1.} {\it Suppose that the planar multiply connected domain $\Omega$ with the following two boundaries:
 $$\gamma_I:=\Big\{r(\theta)=R_1+\sin (N_1\theta)\Big\},~~\gamma_E:=\Big\{r(\theta)=R_2+\sin (N_2\theta)\Big\},$$
 where we describe the curves $\gamma_I$ and $\gamma_E$ by using polar coordinates $(r,\theta)$, $R_1$, $R_2>1$ and $N_1,N_2$ are integers bigger than 1.
 We assume that $R_2 > R_1 + 2$, then $\gamma_I$ and $\gamma_E$ are disjoint. Let
$$u(x_1,x_2):=\log\sqrt{{x_1}^2+{x_2}^2},$$
which is obviously a harmonic function in $\Omega$. Note that $z_2:=\min_{\gamma_E}u = \log\sqrt{R_2-1}> Z_1:=\max_{\gamma_I}u=\log\sqrt{R_1+1}$, which is a condition demanded in Theorem \ref{th1.2}. Moreover, $u|_{\gamma_I}$ and $u|_{\gamma_E}$ has only $N_1$ equal maxima and $N_2$ equal minima respectively, which are not the local maxima and minima relative to $\overline{\Omega}$. However, $u$ has no critical points in $\Omega.$}\vspace{0.2cm}

\vspace{0.0cm}\noindent {\bf Counterexample 2.} {\it Suppose that the planar multiply connected domain $\Omega$ with the following two boundaries:
 $$\gamma_I:=\Big\{r(\theta)=R_1+\sin (N\theta)\Big\},~~\gamma_E:=\Big\{r(\theta)=R_2+\sin (N\theta)\Big\},$$
 where $R_2>R_1>1$ such that $R_2-R_1<2$ and $N$ is an integer bigger than 1. Then $\gamma_I$ and $\gamma_E$ are disjoint. Let
$$u(x_1,x_2):=\log\sqrt{{x_1}^2+{x_2}^2},$$
Note that $z_1<z_2<Z_1<Z_2$, which is a condition demanded in Corollary \ref{co4.1}. Moreover, $u|_{\gamma_I}$ and $u|_{\gamma_E}$ has only $N$ equal maxima and $N$ equal minima respectively, which are not the local maxima and minima relative to $\overline{\Omega}$. However, $u$ has no critical points in $\Omega.$}\vspace{0.3cm}
\end{Remark}

\section{The geometric structure of interior critical zero points}
 ~~~~~~In this section, we will study the geometric structure of interior critical zero points of solutions $u$ to linear elliptic equations with nonhomogeneous Dirichlet boundary conditions in a multiply connected domain $\Omega$.

\subsection{Proof of Theorem \ref{th1.3}}
~~~~~In this subsection, we will investigate the geometric structure of interior critical zero points of a solution $u$ in a planar, bounded, smooth and multiply connected domain $\Omega$ with $u|_{\gamma_I}=H,u|_{\gamma_E}=\psi(x)$ and $\psi$ is sign-changing and has $\widetilde{N}$ zero points on $\gamma_E.$

\begin{proof}[Proof of Theorem \ref{th1.3}]  We need divide the proof into two cases.

Firstly, we should show that $u$ has finite critical zero points in $\Omega,$ denoting by $x_1,\cdots,x_l$, and we set $m_i$ as the multiplicity of corresponding critical zero point $x_i~(i=1,\cdots,l)$. Suppose by contradiction that $u$ has infinite critical zero points in $\Omega.$  According to the results of Lemma 3.1 and Lemma 4.1 in \cite{Deng3}, we know that every non-closed zero level line and $\gamma_E$ have at least one intersection point. In addition, the theorem of Hartman and Wintner \cite{Hartman} shows that the interior critical points of $u$ are isolated. Then there exist infinite non-closed zero level lines across these critical zero points and there are infinite zero points on $\gamma_E$, this contradicts with the assumption of Theorem \ref{th1.3}.

(i) Case $D_1$: If all the critical zero points $x_1,\cdots,x_l$ together with the corresponding zero level lines of $\{x\in\Omega: u(x)=0\}$ clustering round these points form one connected set $\mathrm{C}$. Next we need divide the proof of Case $D_1$ into two situations.

(1) Situation 1: If there exists a simply closed curve $\gamma$ of $\{x\in \Omega : u(x)=0\}$ between $\gamma_I$ and $\gamma_E$ such that $u$ has at least one critical point on $\gamma$, that is $H\neq 0$. Then there are just $\sum_{i = 1}^l {{m_i}}$ non-closed zero level lines across these critical zero points. In addition, when $H\neq 0,$ we know that every non-closed zero level line and $\gamma_E$ must have two intersection points. Then we have
 \begin{equation*}\begin{array}{l}
2( \sum\limits_{i = 1}^l {m_i})\leq \widetilde{N},
\end{array}\end{equation*}
that is
 \begin{equation}\label{5.1}\begin{array}{l}
\sum\limits_{i = 1}^l {m_i}\leq \frac{\widetilde{N}}{2}.
\end{array}\end{equation}

(2) Situation 2: Suppose that there does not exist a simply closed curve $\gamma$ of $\{x\in \Omega : u(x)=0\}$ between $\gamma_I$ and $\gamma_E$ such that $u$ has at least one critical point on $\gamma$. Then there are just $(\sum_{i = 1}^l {{m_i}}+1)$ non-closed zero level lines across these critical zero points.

When $H\neq 0$, we know that every non-closed zero level line and $\gamma_E$ must have two intersection points. Then we have
 \begin{equation*}\begin{array}{l}
2( \sum\limits_{i = 1}^l {m_i}+1)\leq \widetilde{N},
\end{array}\end{equation*}
that is
 \begin{equation}\label{5.2}\begin{array}{l}
\sum\limits_{i = 1}^l {m_i}\leq \frac{\widetilde{N}}{2}-1.
\end{array}\end{equation}

When $H=0,$ by strong maximum principle, we know that above connected set $\mathrm{C}$ and $\gamma_I$ have at most one intersection point. If above connected set $\mathrm{C}$ and $\gamma_I$ do not have intersection points, then every non-closed zero level line and $\gamma_E$ must have two intersection points and (\ref{5.2}) holds. If above connected set $\mathrm{C}$ and $\gamma_I$ have one intersection point, by the results of Lemma 4.1 in \cite{Deng3}, then there must be a independent zero level line connecting $\gamma_I$ and $\gamma_E$, and (\ref{5.2}) still hold.

(ii) Case $D_2$: If all the critical zero points $x_1,\cdots,x_l$ together with the corresponding zero level lines of $\{x\in\Omega: u(x)=0\}$ clustering round these points form $q~(q\geq 2)$ connected sets. Then there are just $(\sum_{i = 1}^l {{m_i}}+q)$ non-closed zero level lines across these critical zero points. Next we need divide the proof of Case $D_2$ into two situations.

(3) Situation 3: When $H\neq 0$, we know that each non-closed zero level line and $\gamma_E$ must have two intersection points. Then we have
 \begin{equation*}\begin{array}{l}
2( \sum\limits_{i = 1}^l {m_i}+q)\leq \widetilde{N},
\end{array}\end{equation*}
that is
 \begin{equation*}\begin{array}{l}
\sum\limits_{i = 1}^l {m_i}\leq \frac{\widetilde{N}}{2}-q.
\end{array}\end{equation*}

(4) Situation 4: When $H=0$, according to the strong maximum principle, we know that each connected set and $\gamma_I$ have at most one intersection point. So we know that the number of intersection point of $q$ connected sets and $\gamma_E$ is at least $2( \sum\limits_{i = 1}^l {m_i}+q)-q.$  Then we have
\begin{equation*}\begin{array}{l}
2( \sum\limits_{i = 1}^l {m_i}+q)-q\leq \widetilde{N},
\end{array}\end{equation*}
that is
 \begin{equation*}\begin{array}{l}
\sum\limits_{i = 1}^l {m_i}\leq \frac{\widetilde{N}-q}{2}.
\end{array}\end{equation*}
Therefore, we have:\\
If $H\neq 0,$ we have
\begin{equation*}\begin{array}{l}
\begin{array}{l}
\sum\limits_{i = 1}^l {{m_i}}\leq \frac{\widetilde{N}}{2},
\end{array}
\end{array}\end{equation*}
else $H=0,$ we have
\begin{equation*}\begin{array}{l}
\begin{array}{l}
\sum\limits_{i = 1}^l {{m_i}}\leq \frac{\widetilde{N}}{2}-1.
\end{array}
\end{array}\end{equation*}
\end{proof}

\subsection{Proof of Theorem \ref{th1.4}}
~~~~~~In this subsection, we will investigate the geometric structure of interior critical zero points of a solution $u$ in a planar, bounded, smooth and multiply connected domain $\Omega$ with $u|_{\gamma_I}=\psi_1(x),u|_{\gamma_E}=\psi_2(x)$, $\psi_1(x)$ and $\psi_2(x)$ are sign-changing and has $\widetilde{N}_1$ zero points and $\widetilde{N}_2$ zero points on $\gamma_I$ and $\gamma_E$ respectively.

\begin{proof}[Proof of Theorem \ref{th1.4}] The proof of finiteness of critical zero points is same as Theorem \ref{th1.3}, denoting by $x_1,\cdots,x_l,$ and we set $m_i$ as the multiplicity of corresponding critical zero point $x_i$. Next we need divide the proof into two cases.

(i) Case $E_1$: If all the critical zero points $x_1,\cdots,x_l$ together with the corresponding zero level lines of $\{x\in\Omega: u(x)=0\}$ clustering round these points form one connected set and there exists a simply closed curve $\gamma$ of $\{x\in \Omega : u(x)=0\}$ between $\gamma_I$ and $\gamma_E$ such that $u$ has at least two critical points on $\gamma$. Then there are just $\sum_{i = 1}^l {{m_i}}$ non-closed zero level lines across these critical zero points. In addition, by the results of Lemma \ref{le2.2}, we know that every non-closed zero level line and boundaries must have two intersection points. Then we have
 \begin{equation*}\begin{array}{l}
2( \sum\limits_{i = 1}^l {m_i})\leq \widetilde{N_1}+\widetilde{N_2},
\end{array}\end{equation*}
that is
 \begin{equation*}\begin{array}{l}
\sum\limits_{i = 1}^l {m_i}\leq \frac{\widetilde{N_1}+\widetilde{N_2}}{2}.
\end{array}\end{equation*}

(ii) Case $E_2$: If Case $E_1$ does not occur. Then there are at least $(\sum_{i = 1}^l {{m_i}}+1)$ non-closed zero level lines across these critical zero points. According to the analysis of above Case $E_1$, then we have
 \begin{equation*}\begin{array}{l}
2( \sum\limits_{i = 1}^l {m_i}+1)\leq \widetilde{N_1}+\widetilde{N_2},
\end{array}\end{equation*}
that is
 \begin{equation*}\begin{array}{l}
\sum\limits_{i = 1}^l {m_i}\leq \frac{\widetilde{N_1}+\widetilde{N_2}}{2}-1.
\end{array}\end{equation*}
\end{proof}

According to Theorem \ref{th1.3} and Theorem \ref{th1.4}, we can easily have the following result.

\begin{Remark}\label{re5.1}
 Let $\Omega$ be a bounded, smooth and simply connected domain in $\mathbb{R}^{2}$. Suppose that $\psi(x)\in C^1(\overline{\Omega})$ is sign-changing and that $\psi$ has $\widetilde{N}$ zero points on $\partial\Omega.$ Let $u$ be a non-constant solution of the following boundary value problem
\begin{equation}\label{5.3}\begin{array}{l}
\left\{
\begin{array}{l}
\sum\limits_{i,j=1}^{2}a_{ij}(x)u_{x_ix_j}(x)+\sum\limits_{i=1}^{2}b_{i}(x)u_{x_i}(x)=0~~~\rm{in}~\Omega,\\
u=\psi(x)~~~\rm{on}~ \partial\Omega.
\end{array}
\right.
\end{array}\end{equation}
Then $u$ has finite interior critical zero points, denoting by $x_1,\cdots,x_l,$ and
\begin{equation}\label{5.4}\begin{array}{l}
\begin{array}{l}
\sum\limits_{i = 1}^l {{m_i}}\leq \frac{\widetilde{N}}{2}-1,
\end{array}
\end{array}\end{equation}
where $m_i$ is as in Theorem \ref{th1.3}.
\end{Remark}

\end{document}